\documentclass{amsart}

\usepackage{amsfonts,amsmath,amsthm,amssymb,amscd,latexsym,enumerate,etoolbox,mathrsfs,comment}

\usepackage[all,cmtip,pdf]{xy}
\usepackage{thmtools, thm-restate}
\usepackage{tikz-cd}
\usepackage{tikz}
\usetikzlibrary{calc}

\usetikzlibrary{shapes,arrows,decorations.markings}

\tikzset{->-/.style={decoration={
  markings,
  mark=at position .525 with {\arrow{Straight Barb}}},postaction={decorate}}}
\tikzset{->>-/.style={decoration={
  markings,
  mark=at position 0.5 with {\arrow{Straight Barb}},
  mark=at position 0.6 with {\arrow{Straight Barb}}},postaction={decorate}}}

\newcommand{\C}{\mathbb{C}} 
\newcommand{\N}{\mathbb{N}}

\newcommand{\Z}{{\mathbb Z}}

\renewcommand{\S}{\mathscr{S}}

\renewcommand{\phi}{\varphi}

\newcommand{\Per}{\text{Per}}

\newcommand{\interior}{\text{int}}

\newcommand{\parens}[1]{\left( #1 \right)}

\newcommand{\word}[1]{\mathtt{#1}}

\newcommand{\lc}{\text{lc}}
\newcommand{\lcu}{\text{lcu}}
\newcommand{\lcs}{\text{lcs}}
\newcommand{\sync}{\text{sync}}

\theoremstyle{plain}
    \newtheorem{theorem}{Theorem}[section]
    \newtheorem{lemma}[theorem]{Lemma}
    \newtheorem{corollary}[theorem]{Corollary}
    \newtheorem{proposition}[theorem]{Proposition}
    \newtheorem{prop}[theorem]{Proposition}
    \newtheorem{conjecture}[theorem]{Conjecture}
\theoremstyle{definition}
    \newtheorem{definition}[theorem]{Definition}
    \newtheorem{example}[theorem]{Example}

    \newtheorem{remark}[theorem]{Remark}
    
\theoremstyle{remark}

\begin{document}

\title[Finitely presented systems and Ruelle algebras]{Synchronizing Dynamical Systems: Finitely presented systems and Ruelle algebras}
\author{Robin J. Deeley}
\address{Robin J. Deeley,   Department of Mathematics,
University of Colorado Boulder
Campus Box 395,
Boulder, CO 80309-0395, USA }
\email{robin.deeley@colorado.edu}
\author{Andrew M. Stocker}
\address{Andrew M. Stocker,   Department of Mathematics,
University of Colorado Boulder
Campus Box 395,
Boulder, CO 80309-0395, USA }
\email{andrew.stocker@colorado.edu}
\thanks{This work was partially supported by the National Science Foundation under Grants No. DMS 2000057 and 2247424 and Simons Foundation Gift MP-TSM-00002896.}
\begin{abstract}
The main goals of the present paper are to determine the structure of the $C^\ast$-algebras associated to a finitely presented system and to develop the basic theory of the Ruelle algebras associated to a general synchronizing system. The later is related to the former in the sense that we show that Ruelle algebras associated to a finitely presented system are explicitly related to the Smale space case. The relevant $C^\ast$-algebras are the synchronizing heteroclinic algebras that were introduced in our previous work on synchronizing systems. They are very much related to previous work of Thomsen, who in turn was building on work of Ruelle, Putnam, and Spielberg. 
\end{abstract}

\maketitle

\section*{Introduction}
This is the third paper in a series on the structure of $C^\ast$-algebras associated to synchronizing systems. In the first paper of the series \cite{DeeSto}, the definition and fundamental properties of synchronizing systems and their $C^\ast$-algebras were discussed. In the second \cite{DeeStoShift}, the special case of synchronizing shifts was developed in general and through many examples. Those papers, like the current one, very much build on work of Thomsen, Ruelle, Putnam, and Spielberg. Throughout the paper $X$ denotes a compact metric space, $\varphi: X \rightarrow X$ is a homeomorphism, and the pair $(X, \varphi)$ is called a dynamical system. 

The present paper deals with the structure of the $C^\ast$-algebras associated to a finitely presented system and the Ruelle algebras for general synchronizing systems. Finite presented systems were introduced by Fried in \cite{fried1987}. By definition a finitely presented system is an expansive systems that is a factor of a shift of finite type (see Definition \ref{def:finPreSys} below and \cite{fried1987} for more details). Finitely presented systems are a natural generalization of Smale spaces. A Smale space is a dynamical system with a uniformly hyperbolic structure that is encoded using a map called the bracket map. They were introduced by Ruelle \cite{ruelle_2004}. The bracket map gives every point in the space a local product structure (for the precise definition of a Smale space, see Definition \ref{SmaSpaDef} below). A further generalization (of both Smale spaces and finitely presented systems) are synchronizing systems. A synchronizing system like a finitely presented system is not uniformly hyperbolic, but there is an open dense set of points that has a local product structure. Unlike finitely presented systems there is no assumption on the existence of a factor map from a shift of finite type. In summary, we have the following implications:
\[ (X,\varphi) \text{ is a Smale space } \implies (X,\varphi) \text{ is finitely presented } \implies (X,\varphi) \text{ is synchronizing } \] 
For the reader familiar with shift spaces, we note that a shift is a Smale space if and only if it is of finite type and a shift is finitely presented if and only if it is sofic. The interested reader can see \cite{DeeStoShift} (among other sources) for many examples of shifts that are synchronizing, but not sofic. As such, the reverse of the implications above do not hold. 

It is therefore of interest to determine what results generalizes from the Smale space case to the finitely presented case and in turn what results generalize from the finitely presented case to the synchronizing case. We are particularly interested in results concerning the structure of $C^*$-algebras associated to these dynamical systems.

A summary of the results of the paper will now be discussed. Although it is known that the theory of the Ruelle algebras generalizes from the Smale space case (see for example \cite[Introduction]{thomsen2010c}) it has never been developed in detail. As such, we introduce the Ruelle algebras for synchronizing systems. Here we follow the work of Putnam and Spielberg who considered the Smale space case in \cite{putnam99} along with following the work of Thomsen in \cite{thomsen2010c}. 

Next, we restrict our attention to finitely presented systems. Todd Fisher \cite{fisher13} has proved that an irreducible finitely presented system has two covers by Smale spaces where the factor maps are particularly nice. For one cover it is $s$-resolving, for the other it is $u$-resolving and both are minimal in a natural sense that we will discuss in below, see in particular Definition \ref{def:Ures} and Theorem \ref{theorem:fp-factor-of-smale}; the interested reader can also find more details and context in the introduction of \cite{fisher13}. 

We prove that the $C^\ast$-algebras associated to a finitely presented system are naturally related to the ones associated to these Smale spaces that cover it in this particularly nice way. The precise results occurs at the groupoid level and is the main result of the paper, see Theorem \ref{thm:main}. The statement of this theorem is involved so we do not state it here in the introduction. Needless to say, the algebras associated to a finitely presented systems have very much the same structure as the Smale space case and we can transfer results and $K$-theory computations from the Smale space case to finitely presented case in a more or less direct way. In particular, a lot of work has been done on computing the $K$-theory of Smale space $C^*$-algebras, see for example \cite{DeeYas_2020, ProYam_2022, ProYam_2023, ProYam_preprint3, ProYam_preprint4} and references therein.

On the other hand, many of the interesting results about Smale space $C^\ast$-algebras are based on the interactions between these algebras. A prototypical example is the follows: 
\begin{conjecture}
Suppose that $(X, \varphi)$ is a mixing Smale space and $S$ and $U$ denote its stable and unstable algebras respectively. Then 
\[
{\rm rank}(K_*(S))={\rm rank}(K_*(U)).
\]
\end{conjecture}
It is worth noting that this conjecture holds for all shifts of finite type and we believe that it will hold for all mixing Smale spaces, see \cite{ProYam_2022, ProYam_2023, ProYam_preprint3, ProYam_preprint4} for more details.

While it is natural to try to generalize this conjecture to the finitely presented case, it fails to do so. We gave an example in \cite{DeeStoShift}, but there are in fact many such examples. The reason that such examples are possible is the asymmetry in the minimal $s$-resolving and $u$-resolving covers of a finitely presented system.

The paper is organized as follows. The preliminaries section deals with expansive dynamical systems, including the notion of a local product structure for a point, metrics on expansive systems, and factor maps between such systems. In Section \ref{SecGroupoid}, we introduce the groupoids associated to a synchronizing system building on work of Thomsen (as well as Ruelle, Putnam, and Spielberg). From these groupoids one obtained $C^*$-algebras, which is the topic of Section \ref{c-star-algebras}. In Section \ref{sec:Ruelle}, the theory of Ruelle algebras for synchronizing systems is developed in detail. The main results of the paper, see Section \ref{sec:Main}, center on the groupoids and $C^*$-algebras of finitely presented systems. The most important result is Theorem \ref{thm:main}. A special case of this main result (namely, the case of sofic shifts) appeared in \cite{DeeStoShift}. The proof in the general case is much more involved. 

\section*{Acknowledgments}

The authors thank Ian Putnam for interesting and insightful discussions.

\section{Preliminaries}

\subsection{Expansive dynamical systems}
In our context, a dynamical system is a pair $(X, \varphi)$ where $X$ is a compact metric space and $\varphi: X \rightarrow X$ is a homeomorphism. We will typically assume that $X$ is infinite. If $x\in X$, then the orbit of $x$ is 
\[ \{ \varphi^n(x) \mid n \in \mathbb{Z} \}.  \]
The forward orbit of $x$ is $\{ \varphi^n(x) \mid n \ge 0 \}$ and likewise for the backward orbit of $x$ only with $n\le 0$.
\begin{definition}\label{def:factor}
A \emph{factor map} between the dynamical systems $(X,\varphi)$ and $(Y, f)$ is a surjective continuous map $\pi : X \to Y$ such that $\pi \circ \varphi  = f\circ \pi$. In this case, we then say that $(Y, f)$ is a \emph{factor} of $(X,\varphi)$ and that $(X, \varphi)$ is an extension or cover of $(Y, f)$.  If $\pi$ is also a homeomorphism then we say that $(X,\varphi)$ and $(Y, f)$ are \emph{(topologically) conjugate}.
\end{definition}
\begin{definition}\label{def:expansiveness}
We say that $(X,\varphi)$ is \emph{expansive} if there is a constant $\varepsilon_X > 0$ (called an \emph{expansiveness constant} of $(X,\varphi)$) such that for any $x,y \in X$, \[ d(\varphi^n(x),\varphi^n(y)) \leq \varepsilon_X \text{ for all } n \in \mathbb{Z} \] implies $x = y$.
\end{definition}
\begin{definition}\label{def:periodic}
Suppose $(X,\varphi)$ is a dynamical system.  A \emph{periodic point} is a point $p \in X$ such that $\varphi^n(p) = p$ for some $n \geq 1$.  If $n = 1$ then $p$ is a \emph{fixed point}.  We define \[ \Per_n(X,\varphi) = \{ p \in X \mid \varphi^n(p) = p \} \] for each $n \geq 1$.  We also have the set of all periodic points: 
\[ \Per(X,\varphi) = \bigcup_{n \geq 1} \Per_n(X,\varphi). \ \]
\end{definition}
We will use the following basic fact: if $(X,\varphi)$ and $(Y,f)$ are dynamical systems and $\pi : X \to Y$ is a factor map, then $\pi\parens{\Per(X,\varphi)} \subseteq \Per(Y,f)$.  

\begin{definition} \label{def:Irreducible}
A dynamical system $(X,\varphi)$ is called \emph{irreducible} if for any two non-empty open sets $U, V \subseteq X$ there is an $n > 0$ such that $\varphi^n(U) \cap V \neq \emptyset$.
\end{definition}

\begin{definition}
A dynamical system $(X,\varphi)$ is called \emph{mixing} if for any pair of non-empty open sets $U, V \subseteq X$ there is an $N > 0$ such that $\varphi^n(U) \cap V \neq \emptyset$ for all $n \geq N$.  
\end{definition}

Notice that $(X,\varphi)$ is mixing implies $(X,\varphi)$ is irreducible, but the converse does not hold in general.  

There are several equivalence relations that are related to the asymptotic behavior of dynamical systems and will be important when constructing $C^\ast$-algebras.
\begin{definition}
Suppose $(X,\varphi)$ be a dynamical system, then we have the following equivalence relations on $X$.
\begin{enumerate}[(i)]
    \item $x \sim_\text{s} y$ when $\displaystyle \lim_{n\to\infty} d(\varphi^n(x), \varphi^n(y)) = 0$,
    \item $x \sim_\text{u} y$ when $\displaystyle \lim_{n\to\infty} d(\varphi^{-n}(x), \varphi^{-n}(y)) = 0$, and
    \item $x \sim_\text{h} y$ when $x \sim_\text{s} y$ and $x \sim_\text{u} y$.
\end{enumerate}
If $x \sim_\text{s} y$ we say that $x$ and $y$ are \emph{stably equivalent} and likewise if $x \sim_\text{u} y$ we say that $x$ and $y$ are \emph{unstably equivalent}.  If $x \sim_\text{h} y$ we say that $x$ and $y$ are \emph{homoclinic}.  The stable, unstable, and homoclinic equivalence classes of $x$ are denote by $X^\text{s}(x)$, $X^\text{u}(x)$, and $X^\text{h}(x)$, respectively.  If $P \subseteq X$, then we will also denote $X^\text{s}(P)$ as the set \[ X^\text{s}(P) = \bigcup_{p \in P} X^\text{s}(p) \] and likewise for $X^\text{u}(P)$.
\end{definition}

The next lemma is well-known. A proof of the first can be found \cite{DeeSto}.
\begin{lemma}\label{lem:periodic-h-implies-equals}
If $(X,\varphi)$ is an expansive dynamical system and $p , q \in \Per(X,\varphi)$ are periodic points such that either $p \sim_\text{s} q$ or $p \sim_\text{u} q$, then $p = q$.  In particular $p \sim_\text{h} q$ implies $p = q$.
\end{lemma}

\begin{definition} \label{SmaSpaDef}
A Smale space is a metric space $(X, d)$ along with a homeomorphism $\varphi: X\rightarrow X$ with the following additional structure: there exists global constants $\epsilon_X>0$ and $0< \lambda < 1$ and a continuous map, called the bracket map, 
\[
[ \ \cdot \  , \ \cdot \ ] :\{(x,y) \in X \times X : d(x,y) \leq \epsilon_X\}\to X
\]
such that the following axioms hold
\begin{itemize}
\item[B1] $\left[ x, x \right] = x$;
\item[B2] $\left[x,[y, z] \right] = [x, z]$ when both sides are defined;
\item[B3] $\left[[x, y], z \right] = [x,z]$ when both sides are defined;
\item[B4] $\varphi[x, y] = [ \varphi(x), \varphi(y)]$ when both sides are defined;
\item[C1] For $x,y \in X$ such that $[x,y]=y$, $d(\varphi(x),\varphi(y)) \leq \lambda d(x,y)$;
\item[C2] For $x,y \in X$ such that $[x,y]=x$, $d(\varphi^{-1}(x),\varphi^{-1}(y)) \leq \lambda d(x,y)$.
\end{itemize}
We denote a Smale space simply by $(X,\varphi)$.
\end{definition}

An introduction to Smale spaces can be found in \cite{ruelle_2004}. In particular, every Smale spaces is expansive and as the notation suggests $\epsilon_X$ is an expansiveness constant for a Smale space. However, there are (many) expansive systems that are not Smale spaces. In particular, a shift space is a Smale space if and only if it is a shift of finite type, see \cite[Theorem 2.2.8]{MR3243636}, but there are (many) shift spaces that are not of finite type.

Another important class of expansive dynamical systems are the finitely presented ones. The definition of finitely presented systems is due to Fried \cite{fried1987} where several equivalent characterizations are discussed. We will use the following one: 

\begin{definition} \label{def:finPreSys}
An expansive dynamical system $(X,\varphi)$ is called \emph{finitely presented} if it is a factor of a shift of finite type.  
\end{definition}

Bowen's theorem \cite{bowen} implies that every Smale space is finitely presented. The shift spaces that are finitely presented systems are exactly the sofic shifts. There are many more examples of finitely presented systems that are not Smale spaces. We give a few examples, see \cite{fried1987} for more examples and details. 

\begin{example} \label{evenShiftExample}
The even shift is the sofic shift space in the alphabet $\{\word{0}, \word{1}\}$ obtained from the following labelled graph: 

\vspace{1em}
\begin{center}
\begin{tikzpicture}
\tikzset{every loop/.style={looseness=10, in=130, out=230}}
\pgfmathsetmacro{\S}{0.9}

\node[shape=circle, draw=black] (A) at (0,0) {};
\node[shape=circle, draw=black] (B) at (1,0) {};
\path[->,>=stealth] (A) edge[loop left] node[scale=\S, left] {$\word{1}$} (A);
\path[->,>=stealth] (A) edge[out=60, in=120] node[scale=\S, above] {$\word{0}$} (B);
\path[->,>=stealth] (B) edge[out=240, in=300] node[scale=\S, below] {$\word{0}$} (A);
\end{tikzpicture}    
\end{center}
Let $X_G$ be the edge shift coming from the unlabeled graph above. Define a map, $\pi$, by taking the map at the shift level induced from mapping defined by taking an unlabeled edge to its label. Then $\pi$ is a factor map from $X_G$ to the even shift.
\end{example}

\begin{example}
Pseudo Anosov diffeomorphisms on closed manifolds form an important class of finitely presented systems. For each orientable surface of genus greater than one, there is a pseudo Anosov diffeomorphism that is not a Smale space. Fried discussing this class of example in more detail in \cite{fried1987}.
\end{example}

\begin{example}
Certain solenoids (stationary inverse limits) are finitely presented systems. For a reader familiar with the one-dimensional Smale space solenoids considered in \cite{WilOneDim, Wil, Yi2001, Yi2003}, we present an explicit example. 

Let $Y$ be the wedge of two circles:

\begin{figure}[h]
\begin{tikzpicture}
\draw (0,0) circle [radius=3];
\draw (0,-1) circle [radius=2];
\draw[fill] (0,-3) circle [radius=0.1];

\draw (-1.3,0.7) -- (-1.15, 0.5);
\draw (1.3,0.7) -- (1.15, 0.5);
\draw (-2.42487,1.4) -- (-2.77128,1.6);
\draw (2.42487,1.4) -- (2.77128,1.6);
\node [below] at (0,-3.1) {$p$};
\node [above] at (0,3.1) {$a$};
\node [left] at (-2.77128,-1.6) {$a$};
\node [right] at (2.77128,-1.6) {$b$};
\node [right] at (-1.9,-1) {$a$};
\node [above] at (0, 1.1) {$b$};
\node [left] at (1.9,-1) {$a$};
\end{tikzpicture}
\caption{The wedge of two circles and labels determining $g$}
\label{Figure-aab/ab-PreSolenoid}
\end{figure}
A map, $g : Y \rightarrow Y$ is defined using the labels as follows. To begin, let the outer circle be the $a$-circle and the inner circle be the $b$-circle. Then each line segment labelled with $a$ is mapped onto the $a$-circle (i.e., the outer circle) and each line segment labelled with $b$ is mapped onto the $b$-circle (i.e., the inner circle). This is done in an orientation-preserving way (we oriented both circles the same way, say clockwise). 

We remark that $g$ is not a local homeomorphism and that $g$ does not satisfy the Flattening Axiom in \cite{WilOneDim, Yi2001}. The associated solenoid is the pair $(X, f)$ where 
\[ X= \{ (y_0, y_1, y_2, \ldots ) \mid g(y_{i+1})=y_i \hbox{ for i} \in \N \}
\]
and
\[ f(y_0, y_1, y_2, \ldots )=( g(y_0), g(y_1), g(y_2), \ldots) \]
The system $(X, \varphi)$ is finitely presented, but not a Smale space. To see that it is finitely presented, one constructions a factor map from the shift of finite type associated to the graph

\vspace{1em}
\begin{center}
\begin{tikzpicture}
\tikzset{every loop/.style={looseness=9, in=130, out=230}}
\pgfmathsetmacro{\S}{0.9}

\node[shape=circle, draw=black] (A) at (0,0) {};
\node[shape=circle, draw=black] (B) at (1,0) {};
\path[->,>=stealth] (A) edge[loop left] (A);
\path[->,>=stealth] (A) edge[looseness=18, in=110, out=250] (A);
\path[->,>=stealth] (B) edge[looseness=10, out= 320, in=50] (B);
\path[->,>=stealth] (A) edge[out=60, in=120] (B);
\path[->,>=stealth] (B) edge[out=240, in=300] (A);
\path[->,>=stealth] (B) edge[out=210, in=330] (A);

\end{tikzpicture}    
\end{center}

following the process discussed in detail in \cite[Section 5]{Yi2001}. To see that it is not a Smale space, one uses the fact that the Flattening Axiom fails for $(Y, g)$. This example and other similar examples will be studied in future work.
\end{example}

\subsection{Local Product Structure}\label{sec:local-stable-unstable}

In this section we discuss the local stable and unstable sets of a point in an expansive dynamical system.  

Let $(X, \varphi)$ be an expansive dynamical system with expansiveness constant $\varepsilon_X > 0$.  In the rest of this section we follow \cite{fried1987}.
\begin{definition}\label{def:loc-stable-unstable-sets}
For each $x \in X$ and $0<\varepsilon < \varepsilon_X$ we define the following subsets called, respectively, the \emph{local stable} and \emph{local unstable} sets of $x$.
\begin{align*}
    X^\text{s}(x, \varepsilon) &= \left\{ y \in X \mid d\left(\varphi^n(x),\varphi^n(y)\right) < \varepsilon \text{ for all } n \geq 0 \right\} \\
    X^\text{u}(x, \varepsilon) &= \left\{ y \in X \mid d\left(\varphi^{-n}(x),\varphi^{-n}(y)\right) < \varepsilon \text{ for all } n \geq 0 \right\}
\end{align*}
Furthermore, one can show directly from this definition that
\begin{align*}
    \varphi^{-N}\left(X^\text{s}\left(\varphi^{N}(x), \varepsilon\right)\right) &= \left\{ y \in X \mid d\left(\varphi^n(x),\varphi^n(y)\right) < \varepsilon \text{ for all } n \geq N \right\} \text{ , and } \\
    \varphi^{N}\left(X^\text{u}\left(\varphi^{-N}(x), \varepsilon\right)\right) &= \left\{ y \in X \mid d\left(\varphi^{-n}(x),\varphi^{-n}(y)\right) < \varepsilon \text{ for all } n \geq N \right\}
\end{align*}
for any $N \in \Z$.
\end{definition}
In the Smale space case (see \cite[Definition 2.1.7]{MR3243636}) we have that 
\begin{align*}
    X^\text{s}(x, \varepsilon) &= \left\{ y \in X \mid [x, y]=y \hbox{ and }d(x,y)< \varepsilon \right\} and \\
    X^\text{u}(x, \varepsilon) &= \left\{ y \in X \mid [x,y]=x \hbox{ and }d(x,y)< \varepsilon \right\}.
\end{align*}

The next few lemmas (and definitions) will be useful. They are either well-known (e.g., Lemma \ref{lem:bracket-unique-intersection}) or can be found, with proofs in the case of the lemmas, in \cite{DeeSto}.

\begin{lemma}\label{lem:bracket-unique-intersection}
For $\displaystyle 0 < \varepsilon \leq \frac{\varepsilon_X}{2}$, the intersection $X^\text{s}(x, \varepsilon) \cap X^\text{u}(y, \varepsilon)$ consists of at most one point in $X$.
\end{lemma}

\begin{definition}\label{def:bracket-map}
For $0 < \varepsilon \leq \frac{\varepsilon_X}{2}$ we define the set
\[ D_\varepsilon = \{ (x, y) \in X \times X \mid X^\text{s}(x, \varepsilon) \cap X^\text{u}(y, \varepsilon) \neq \emptyset \} \]
and a map $[-,-]: D_\varepsilon \to X$, called the \emph{bracket map}, such that $[x,y] \in X^\text{s}(x, \varepsilon) \cap X^\text{u}(y, \varepsilon)$.  By Lemma \ref{lem:bracket-unique-intersection} this map is well-defined.
\end{definition}
\begin{remark}
It is well-known that if $(X, \varphi)$ is a Smale space, then the definition of the bracket map in the previous definition agrees with the one in the definition of a Smale space.
\end{remark}

\begin{lemma}\label{lem:epsilon-naught-existence}
Suppose $(X,\varphi)$ is an expansive dynamical system. There exists $\varepsilon_0 > 0$ such that for $\displaystyle 0 < \varepsilon, \varepsilon' \leq \varepsilon_0$, $(x,x)$ is in the interior of $D_\varepsilon$ if and only if it is in the interior of $D_{\varepsilon'}$.
\end{lemma}

\begin{definition}\label{def:synchronizing-point}
Suppose $(X,\varphi)$ is an expansive dynamical system.  A point $x \in X$ is called \emph{synchronizing} if $(x,x)$ is in the interior of $D_\varepsilon$ for some $\displaystyle 0 < \varepsilon \leq \varepsilon_0$, where $\varepsilon_0$ is as in Lemma \ref{lem:epsilon-naught-existence}.  We will denote the subset of synchronizing points $X_\sync$.
\end{definition}

\begin{lemma} \label{lem:DclosedBraCts}
The set $D_\varepsilon$ is closed in $X \times X$ and the map $[-,-]$ is (uniformly) continuous.  Additionally, the set $\Delta_X = \{(x,x) \mid x \in X\}$ is contained in $D_\varepsilon$, and $[x,x] = x$ for all $x \in X$.
\end{lemma}

\begin{lemma}
Suppose $0< \epsilon < \frac{\varepsilon_X}{2}$. If $z \in X^s(y, \epsilon)$, then $[y,z]=z$ and $[z,y]=y$.
\end{lemma}

\begin{lemma}\label{lem:stable-unstable-inductive-limit-top}
Let $(X,\varphi)$ be an expansive dynamical system and let $x \in X$.  Then for any $\varepsilon > 0$ the stable and unstable equivalence classes satisfy
\begin{enumerate}[(i)]
    \item $\displaystyle X^\text{s}(x) = \bigcup_{N \geq 0} \varphi^{-N}\left(X^\text{s}\left(\varphi^{N}(x), \varepsilon\right)\right)$, and
    \item $\displaystyle X^\text{u}(x) = \bigcup_{N \geq 0} \varphi^{N}\left(X^\text{u}\left(\varphi^{-N}(x), \varepsilon\right)\right)$,
\end{enumerate}
respectively.  Hence we will often refer to the equivalence classes $X^\text{s}(x)$ and $X^\text{u}(x)$ as the \emph{global stable set} and \emph{global unstable set} of $x$, respectively.
\end{lemma}

When $p$ is a periodic point, we topologize $X^\text{s}(p)$ and $X^\text{u}(p)$ with the inductive limit topology coming from Lemma \ref{lem:stable-unstable-inductive-limit-top}, see \cite[Section 4.1]{thomsen2010c}.  Note that with this topology $X^\text{s}(p)$ and $X^\text{u}(p)$ are both locally compact and Hausdorff.

\begin{definition}\label{def:synchronizing-general}
An expansive dynamical system $(X,\varphi)$ is called \emph{synchronizing} if it is irreducible and it has at least one synchronizing point.
\end{definition}

\subsection{Metrics on expansive systems} \label{sec:MetExpSys}
Given an expansive dynamical system, Fried proved in \cite[Lemma 2]{fried1987} that there exists a metric $d$ (we will call $d$ an \emph{adapted metric}) and constants $\eta > 0$, $0 < \lambda < 1$ such that $d$ is gives the original topology on $X$,
\begin{align*}
    d(\varphi(x), \varphi(y)) &\leq \lambda d(x, y) \text{ for all } y \in X^\text{s}(x, \eta) \text{ and} \\
    d(\varphi^{-1}(x), \varphi^{-1}(y)) &\leq \lambda d(x, y) \text{ for all } y \in X^\text{u}(x, \eta) \,.
\end{align*}
In addition, by \cite{sakai01} this metric can additionally be taken so that both $\varphi$ and $\varphi^{-1}$ are Lipschitz for some Lipschitz constant $K > 1$.  

In the Smale space case, we can do even better. Sakai \cite{sakai01} proved the following (also see \cite[Lemma 4.10]{Ger2022} for more details).
\begin{lemma} \label{lem:AXconstant}
Suppose that $(X, \varphi)$ is a Smale space with an adapted metric such that both $\varphi$ and $\varphi^{-1}$ are Lipschitz. Then there exists $A_X>0$ such that 
\[
d(a, [a, b])\leq A_Xd(a,b) \hbox{ and }d(b, [a, b])\leq A_Xd(a,b)
\]
where $a, b \in X$ satisfy $d(a, b)< \epsilon_X$ (so that $[a, b]$ is well-defined). It is useful to note that we can assume that $A_X \ge 1$.
\end{lemma}

\subsection{Maps} \label{SecMaps}
The next lemma is due to Fisher \cite[Lemma 3.3]{fisher13} (also see \cite[Lemma 2.2]{PutLift} in the case when both spaces are Smale spaces).
\begin{lemma} \label{lem:EpsPi}
Suppose $\pi: (X, \varphi) \rightarrow (Y, f)$ is a factor map from a Smale space to a finitely presented system. Then there exists $\epsilon_{\pi}>0$ such that if $x_1, x_2 \in X$ with $d(x_1, x_2)< \epsilon_{\pi}$, then $\pi([x_1, x_2])=[\pi(x_1), \pi(x_2)]$.
\end{lemma}

\begin{definition} \label{def:Ures}
If $\pi:X \to Y$ is a factor map of expansive dynamical systems, then $\pi$ is called \emph{u-resolving} (respectively, \emph{s-resolving}) is $\pi$ restricted to $X^\text{u}(x)$ (respectively, $X^\text{s}(x)$) is injective for all $x \in X$.  We also say that $\pi$ is \emph{one-to-one almost everywhere or almost one-to-one} if there is a residual set of points in $X$ with unique pre-image under $\pi$.  
\end{definition}

We have the following result from \cite{fisher13} with the ``Moreover" part following from \cite[Lemma 6.6]{DeeSto}.   

\begin{theorem}{\cite[Theorem 1.1, Lemma 3.2]{fisher13}}\label{theorem:fp-factor-of-smale}
Suppose $(Y, f)$ is an irreducible finitely presented system.  Then there is an irreducible Smale space $(X,\varphi)$ and a u-resolving factor map $\pi:X \to Y$ such 
\begin{enumerate}
\item $\pi$ is one-to-one almost everywhere, 
\item there exists a dense open set $W \subseteq Y$ such that each periodic point in $W$ has a unique pre-image under $\pi$, and 
\item $\pi: X \rightarrow Y$ is the minimal such extension.
\end{enumerate}
Moreover, if $(Y, f)$ is mixing, then $(X, \varphi)$ is also mixing.
\end{theorem}
\begin{definition}
The Smale space $(X, \varphi)$ and the factor map $\pi$ together are called \emph{the minimal u-resolving extension} (or cover) of $(Y, f)$ and in a similar way there is a minimal s-resolving extension (or cover). 
\end{definition}
When $(Y, f)$ is a sofic shift, the minimal u-resolving cover is called the left Fischer cover and likewise the minimal s-resolving cover is call the right Fischer cover. The construction of these covers is well established, see \cite[Section 3.3]{lindmarcus} for a detailed elementary development. In the general case of irreducible finitely presented systems, the construction is much more involved and recent (as mentioned above it is due to Fisher \cite{fisher13}). These minimal extensions will play an important role in the present paper.

\begin{definition}
An irreducible finitely presented system $(Y, f)$ is called \emph{an almost Smale space} if there exists a Smale space $(X, \varphi)$, and a factor map $\pi : X \rightarrow Y$ such that $\pi$ is both s and u resolving and almost one-to-one. If in this situation $(Y, f)$ is a sofic shift (and hence $(X, \varphi)$ is shift of finite type), then we say $(Y, f)$ is of \emph{almost finite type}.
\end{definition}
The next result is due to Fisher \cite{fisher13}.
\begin{theorem}
Suppose $(Y, f)$ is an irreducible almost Smale space, $(X, \varphi)$ is a Smale space, and $\pi : X \rightarrow Y$ is both $s$ and $u$-resolving and almost one-to-one. Then $\pi$ is both the minimal $s$ and $u$-resolving extension.
\end{theorem}

\begin{example}
Recall that the even shift is the sofic shift space in the alphabet $\{\word{0}, \word{1}\}$ obtained from the following labelled graph: 

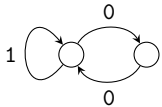
\vspace{1em}
\begin{center}
\begin{tikzpicture}
\tikzset{every loop/.style={looseness=10, in=130, out=230}}
\pgfmathsetmacro{\S}{0.9}

\node[shape=circle, draw=black] (A) at (0,0) {};
\node[shape=circle, draw=black] (B) at (1,0) {};
\path[->,>=stealth] (A) edge[loop left] node[scale=\S, left] {$\word{1}$} (A);
\path[->,>=stealth] (A) edge[out=60, in=120] node[scale=\S, above] {$\word{0}$} (B);
\path[->,>=stealth] (B) edge[out=240, in=300] node[scale=\S, below] {$\word{0}$} (A);
\end{tikzpicture}    
\end{center}
As in Example \ref{evenShiftExample}, we let $X_G$ be the edge shift coming from the unlabeled graph above and define a map, $\pi$, by taking the map at the shift level induced from mapping defined by taking an unlabeled edge to its label. Then $\pi$ is almost one-to-one, $s$-resolving, and $u$-resolving. Hence, the even shift is of almost finite type.

There are many sofic shifts that are not of almost finite type, see for example \cite[Section 13.1]{lindmarcus}.
\end{example}

\section{Groupoids} \label{SecGroupoid}

We introduce several equivalence relations which capture notions of asymptotic equivalence of elements in an expansive dynamical system $(X,\varphi)$.  The definition of local conjugacy (including the stable and unstable versions) is due to Thomsen \cite{thomsen2010c}.
\subsection{The Local Conjugacy Relation}

As usual, $(X,\varphi)$ is an expansive dynamical system.
\begin{definition}
We say that two points $x,y \in X$ are \emph{locally conjugate}, denoted $x \sim_\lc y$, if there exist two open neighborhoods (in $X$) $U$ and $V$ of $x$ and $y$ respectively, and a homeomorphism $\gamma : U \to V$ such that $\gamma(x) = y$ and \[ \lim_{n \to \pm \infty} \sup_{z \in U} \, d(\varphi^n(z), \varphi^n(\gamma(z))) = 0 \,. \]
We will denote the equivalence class of $x$ under the local conjugacy relation as $X^\lc(x)$, and we will also call the triple $(U,V,\gamma)$ a \emph{local conjugacy} from $x$ to $y$.
\end{definition}

One can show that local conjugacy is an equivalence relation. A crucial fact about local conjugacy is that, given two points that are locally conjugate, then a local conjugacy (that is, the map $\gamma$ in the previous definition) between the two points is essentially unique, see \cite[Lemma 1.4]{thomsen2010c} for the precise statement.

It is worth commenting that the definition of local conjugacy (which is due to Thomsen as mentioned above) is based on work of Ruelle \cite{Ruelle1988}, and the fact that the local conjugacy and homoclinic relations agree for Smale spaces. This makes local conjugacy a suitable relation for generalizing Ruelle's and Putnam's $C^\ast$-algebraic constructions for Smale spaces to the class of all expansive dynamical systems.

Two important fact about local conjugacy are that it respects the property of being synchronizing and it is invariant under $\varphi$. These properties are formally stated in the next two results; both were proved in \cite{DeeSto}.

\begin{prop}\label{prop:lc-preserves-sync}
Let $(X,\varphi)$ be an expansive system and $x,y \in X$.  If $x$ is synchronizing and $x \sim_\lc y$, then $y$ is synchronizing.
\end{prop}

\begin{lemma}\label{lem:lc-invariant-under-map}
Let $(X,\varphi)$ be an expansive system and $x,y \in X$.  Then $x \sim_\lc y$ if and only if $\varphi(x) \sim_\lc \varphi(y)$.
\end{lemma}

\subsection{The Stable/Unstable Local Conjugacy Relations}

We also have the notions of local conjugacy that only hold in the limit in one direction.  Similar to local conjugacy, these respectively imply stable and unstable equivalence. The converse is not true; the even shift can be used to show this and the details similar to \cite[Remark 1.13]{thomsen2010c}.

\begin{definition}
Let $(X,\varphi)$ be an expansive dynamical system with $x,y \in X$.  Suppose $x \in U \subseteq X^u(x)$ is open, $y \in V\subseteq X^u(y)$ is open, and $\gamma : U \to V$ is a homeomorphism onto its image such that $\gamma(x) = y$ and \[ \lim_{n \to \infty} \sup_{z \in U} d(\varphi^n(z), \varphi^n(\gamma(z))) = 0 \,. \]  Then $\gamma$ is a called a \emph{stable local conjugacy} and we say that $x \sim_\lcs y$.  We will denote the stable local conjugacy equivalence class of $x$ as $X^\lcs(x)$.
\end{definition}

\begin{definition}
Let $(X,\varphi)$ be an expansive dynamical system with $x,y \in X$.  Suppose $x \in U \subseteq X^s(x)$ is open, $y \in V\subseteq X^s(y)$ is open and $\gamma : U \to V$ is a homeomorphism onto its image such that $\gamma(x) = y$ and \[ \lim_{n \to \infty} \sup_{z \in U} d(\varphi^{-n}(z), \varphi^{-n}(\gamma(z))) = 0 \,. \]  Then $\gamma$ is a called an \emph{unstable local conjugacy} and we say that $x \sim_\lcu y$.  We will denote the ustable local conjugacy equivalence class of $x$ as $X^\lcu(x)$.
\end{definition}

These two relations are indeed equivalence relations, see \cite{DeeSto} for details.  Also notice that $x \sim_\lc y$ implies $x \sim_\lcs y$ and $x \sim_\lcu y$. Both stable and unstable local conjugacy is (in a certain sense) unique, see \cite[Lemma 1.4]{thomsen2010c} for the precise formulation. The reader can find more on these equivalence relations in either \cite{DeeSto} or \cite{thomsen2010c}. 

We will use the following lemma later in the paper; it was proved in \cite{DeeSto}.

\begin{lemma}\label{lem:rect-nbhd-implies-lcu-lcs}
Let $(X,\varphi)$ be a synchronizing system with $x \in X$ synchronizing.  Then if $R$ is a product neighborhood of $x$ and $y \in \interior(R)$, we have
\begin{enumerate}[(i)]
    \item $y \sim_\lcu [x,y] \sim_\lcs x$, and
    \item $y \sim_\lcs [y,x] \sim_\lcu x$.
\end{enumerate}
\end{lemma}

\section{$C^\ast$-Algebras from Expansive Dynamical Systems}\label{c-star-algebras}

In this section we review the construction of $C^\ast$-algebras from the equivalence relations in the previous section. The general idea of this construction can be found in the work of Thomsen \cite{thomsen2010c}, who builds on work of Putnam and Spielberg \cite{putnam_1996, putnam99}. However, there are a few difference, see our previous work in \cite{DeeSto} and Remark \ref{rem:differentalgebras} below.  

The first $C^\ast$-algebra is constructed from the local conjugacy relation for an expansive dynamical system. This construction is due to Thomsen \cite{thomsen2010c}. We first define the groupoid \[ G^\lc(X, \varphi) = \left\{ (x,y) \in X \times X \mid x \sim_\lc y \right\} \] with groupoid composition given by $(x,y)(y',z) = (x,z)$ whenever $y = y'$.  However, instead of topologizing $G^\lc(X, \varphi)$ as a subspace of $X \times X$, the topology is generated from the sets of the form $\{(z,\gamma(z)) \mid z \in U\}$ where $(U, V, \gamma)$ is a local conjugacy. With this topology, $G^\lc(X, \varphi)$ is an \'{e}tale groupoid, see \cite[Theorem 1.7]{thomsen2010c} for details.

\begin{definition}
The \emph{homoclinic algebra} $A(X,\varphi)$ of an expansive dynamical system $(X,\varphi)$ is the reduced $C^\ast$-algebra $C_r^\ast(G^\lc(X, \varphi))$ (where the reduced $C^\ast$-algebra of an \'etale groupoid is defined by Renault in \cite{renault1980groupoid}). 
\end{definition}

If $(X,\varphi)$ is a Smale space, then the $C^\ast$-algebra $A(X,\varphi)$ is exactly the asymptotic (i.e., homoclinic) algebra defined by Putnam in \cite{putnam_1996}. 

\begin{definition}\label{def:g-invariant}
Suppose $G$ is a locally compact Hausdorff \'{e}tale groupoid.  Then $X \subseteq G^0$ is called \emph{$G$-invariant} if for any $\gamma \in G$, $r(\gamma) \in X$ if and only if $s(\gamma) \in X$, 
\end{definition}

\begin{definition}
Let $\mathcal{I}_\text{sync}(X,\varphi)$ be the reduced groupoid $C^\ast$-algebra obtained from 
\[ \{ (x, y) \in X_\text{sync} \times X_\text{sync} \mid x \sim_{lc} y \}. \]
\end{definition}

Using Proposition \ref{prop:lc-preserves-sync} and Definition \ref{def:g-invariant}, we have the following result.

\begin{theorem}\label{theorem:sync-ideal} (\cite[Section 5.4]{DeeSto})
Let $(X,\varphi)$ be a synchronizing system with $A(X,\varphi)$ its homoclinic algebra, then $\mathcal{I}_\text{sync}(X,\varphi) \subseteq A(X,\varphi)$ is an ideal.
\end{theorem}

The next two $C^\ast$ algebras we introduce are called the \emph{synchronizing heteroclinic} algebras. Since we are specializing to the case of synchronizing systems, we follow \cite{DeeSto} where we made some modifications to Thomsen's construction in \cite{thomsen2010c}.  In particular, Thomsen uses the set of all \emph{post-periodic} points, which is defined as the set \[ X^\text{u} = \bigcup_{p \in \Per(X,\varphi)} X^\text{u}(p) \,. \]  This involves using all the periodic points in the dynamical system, and additionally requires the assumption that periodic points are dense.

However, in synchronizing systems, there are many periodic points that are synchronizing and thus behave like the periodic points in a Smale space. Based on this fact, we restrict Thomsen's construction to only these synchronizing periodic points.  In fact, the synchronizing heteroclinic algebras associated to different choices of synchronizing periodic points are all Morita equivalent (see \cite{DeeSto}) which parallels the Smale space result of Putnam and Spielberg \cite{putnam99}. For more details on the construction of these algebras the reader can see \cite{DeeSto}.

\begin{definition}
Let $(X,\varphi)$ be a synchronizing system and let $P \subseteq X$ be a finite set of synchronizing periodic points.  Then we define the \'{e}tale groupoids
\begin{enumerate}[(i)]
    \item $G^\lcs(X, \varphi, P) = \{ (x, y) \in X^\text{u}(P) \times X^\text{u}(P) \mid x \sim_\lcs y \}$, and
    \item $G^\lcu(X, \varphi, P) = \{ (x, y) \in X^\text{s}(P) \times X^\text{s}(P) \mid x \sim_\lcu y \}$.
\end{enumerate}
Note that we are using the inductive limit topology on $X^\text{u}(P)$ and $X^\text{s}(P)$ as discussed just after Lemma \ref{lem:stable-unstable-inductive-limit-top}. These groupoids are topologized from the stable and unstable local conjugacies respectively in the same way as with the homoclinic algebra.  Note that the unit space of $G^\lcs(X, \varphi, P)$ is $X^\text{u}(P)$ and the unit space of $G^\lcu(X, \varphi, P)$ is $X^\text{s}(P)$.
\end{definition}

\begin{definition}
Let $(X,\varphi)$ be a synchronizing system and let $P \subseteq X$ be a finite set of synchronizing periodic points.  Then we define
\begin{enumerate}[(i)]
    \item the \emph{stable synchronizing heteroclinic algebra}, $S(X,\varphi,P) = C^\ast_r(G^\lcs(X, \varphi, P))$, and
    \item the \emph{unstable synchronizing heteroclinic algebra}, $U(X,\varphi,P) = C^\ast_r(G^\lcu(X, \varphi, P))$
\end{enumerate}
which are both reduced groupoid $C^\ast$-algebras in the sense of Renault \cite{renault1980groupoid}.
\end{definition}

In the case of $(X,\varphi)$ being a Smale space, the stable and unstable synchronizing heteroclinic algebras are exactly the stable and unstable algebras defined by Putnam \cite{putnam99}. Like in the Smale space case, we have that $S(X,\varphi,P)$ and $U(X, \varphi,P)$ only depend on $P$ up to Morita equivalence by \cite[Theorem 5.5]{DeeSto}.

\begin{remark} \label{rem:differentalgebras}
The algebras $S(X,\varphi,P)$ and $U(X, \varphi, P)$ are different than the ones Thomsen's constructed in \cite{thomsen2010c}. We call the algebras Thomsen constructed the heteroclinic algebras. The different between our construction and Thomsen's can be see explicitly in the case of the even shift, which is discussed in detail in \cite[Section 5]{DeeStoShift}.
\end{remark}

\section{The Ruelle algebras} \label{sec:Ruelle}

Throughout this section, $(X, \varphi)$ is a synchronizing system. Following \cite{putnam99} (where the Smale space case was considered) we let
\begin{align*}
G^{lcs}(X, \varphi, P) \rtimes \Z & = \{ (x_1, n , x_2) \mid n\in \Z \hbox{ and }(\varphi^n(x_1), x_2) \in G^{lcs}(X, \varphi, P) \} \\
G^{lcu}(X, \varphi, P) \rtimes \Z & = \{ (x_1, n , x_2) \mid n\in \Z \hbox{ and }(\varphi^n(x_1), x_2) \in G^{lcu}(X, \varphi, P) \} \\
G^{lc}(X, \varphi) \rtimes \Z & = \{ (x_1, n , x_2) \mid n\in \Z \hbox{ and }(\varphi^n(x_1), x_2) \in G^{lc}(X, \varphi) \}.
\end{align*}
In each case, the groupoid structure is defined as follows:
\[
(x_1, n , x_2) \cdot (x_1^{\prime}, n^{\prime}, x_2^{\prime}) = \left\{ \begin{array}{cc} (x_1, n+n^{\prime}, x_2^{\prime}) & x_2=x_1^{\prime} \\ \hbox{not defined} & \hbox{ otherwise} \end{array} \right.
\]
and
\[
(x_1, n, x_2)^{-1}= (x_2, -n, x_1)
\]
The synchronizing ideal, $\mathcal{I}_\text{sync}(X,\varphi)$, is invariant under the action induced by $\varphi$, so we get an ideal $\mathcal{I}_\text{sync}(X,\varphi) \rtimes \Z \subseteq C^*_r(G^{lc}(X, \varphi) \rtimes \Z)$.
\begin{definition}
Suppose that $(X, \varphi)$ is a mixing synchronizing system. The $C^*$-algebras 

\end{definition}
\begin{proposition}
Suppose that $(X, \varphi)$ is a mixing synchronizing system. Then the $C^*_r(G^{lcs}(X, \varphi, P) \rtimes \Z)$, $C^*_r(G^{lcu}(X, \varphi, P) \rtimes \Z$, and $C^*_r(G^{lc}(X, \varphi) \rtimes \Z)$ are called the Ruelle algebras associated to the system $(X, \varphi)$.
\begin{enumerate}
\item if the conditions in Lemma 4.19 in \cite{DeeSto} hold, then $G^{lcs}(X, \varphi, P) \rtimes \Z$ and $G^{lcu}(X, \varphi, P) \rtimes \Z$ are minimal; 
\item $G^{lcs}(X, \varphi, P) \rtimes \Z$ and $G^{lcu}(X, \varphi, P) \rtimes \Z$ are essentially principal;
\item if $G^{lcs}(X, \varphi, P)$ and $G^{lcu}(X, \varphi, P)$ are amenable, then $G^{lcs}(X, \varphi, P) \rtimes \Z$ and $G^{lcu}(X, \varphi, P) \rtimes \Z$ are amenable;
\item if $G^{lc}(X, \varphi)$ is amenable, then $G^{lc}(X, \varphi) \rtimes \Z$ is amenable.
\end{enumerate}
\end{proposition}
\begin{proof}
We only consider the case of $G^{lcs}(X, \varphi, P) \rtimes \Z$ as the other arguments are similar. The minimal statement follows from \cite[Lemma 4.9]{DeeSto} (also see \cite[Lemma 5.1]{putnam99} in the Smale space case). 

For the essentially principal statement, given $(x, 0, x)$ in the unit space of the groupoid $G^{lcs}(X, \varphi, P) \rtimes \Z$, we have that
\[
{\rm Iso}(x, 0, x)= \{ (x, n, x) \in G^{lcs}(X, \varphi, P) \rtimes \Z \}.
\]
Still following \cite[Proof of 1.4 on page 20]{putnam99}, we have that ${\rm Iso}(x,0,x)$ is non-trivial for at most countably many $x$ and hence the groupoid, $G^{lcs}(X, \varphi, P) \rtimes \Z$ is essentially principal. (The reader should note that the proofs of Lemmas 5.2, 5.3, and 5.4 in \cite{putnam99} only use the fact that $(X, \varphi)$ is expansive, not that it is a Smale space).

Finally, the amenable statements follows because amenablity is preserved when taking crossed products by $\Z$.
\end{proof}
\begin{remark}
It is currently unclear if $G^{lcs}(X, \varphi, P)$, $G^{lcu}(X, \varphi, P)$, and $G^{lc}(X, \varphi)$ are amenable in general and hence likewise for the crossed products by $\Z$. 
\end{remark}
\begin{definition}
A topological groupoid $G$ is called \emph{locally contracting} if, for each non-empty open subset $U \subseteq G^{0}$, there exists an open $G$-set, $\Delta$, such that
\[ r(\bar{\Delta}) \subset s(\Delta) \subseteq U \hbox{ and }r(\bar{\Delta}) \neq s(\Delta).
\]
\end{definition}
\begin{proposition}
Suppose that $(X, \varphi)$ is a mixing synchronizing system. Then the groupoids $G^{lcs}(X, \varphi, P) \rtimes \Z$ and $G^{lcu}(X, \varphi, P) \rtimes \Z$ are locally contracting.
\end{proposition}
\begin{proof}
As with the previous result, the proof is essentially the same as one in \cite{putnam99} (in this case Proposition 5.7 of \cite{putnam99}). However, we include all the details for completeness and because of the use of the bracket in the proof in \cite{putnam99}, which requires further justification in the synchronizing system case. As usual with a synchronizing system, we use an adapted metric.

Take $\emptyset \neq U \subseteq X^u(P)$ open. We can assume that $U \subseteq X^u(p)$ for some $p\in P$. Using the fact that $p$ is periodic and synchronizing there exists $n\in \N$ such that 
\[
\varphi^{-n}(p)=p \hbox{ and }\varphi^{-n}(U) \cap X^u(p, \epsilon_p) \neq \emptyset 
\]
where $\epsilon_p>0$ is small enough so that $X^s(p, \epsilon_p) \times X^u(p, \epsilon_p)$ is a local product neighborhood.

Form the open subset of $X$ given by
\[
[ \varphi^{-n}(U) \cap X^u(p, \epsilon_p), X^s(p, \epsilon_p) ].
\]
Since periodic points are dense in $X$, there exists a periodic point $q$ in this open set. In particular, there exists $N\ge 1$ with $\varphi^N(q)=q$. Take $\epsilon>0$ small enough so that 
\[
X^u(q, \epsilon) \subseteq X^s(p, \epsilon_p) \times X^u(p, \epsilon_p) \hbox{ and }[X^u(q, \epsilon), p] \subseteq \varphi^{-n}(U) \cap X^u(p, \epsilon_p).
\]
Since $X$ is an infinite set and $(X, \varphi)$ is mixing, $X^u(q, \epsilon)$ is not discrete. It follows that there exists $m \geq 1$ such that 
\begin{equation} \label{contractingEquation}
\varphi^{-mN}(X^u(q, \epsilon)) \subseteq X^u(q, \lambda^{-mN}\epsilon) \hbox{ and }\overline{\varphi^{-mN}(X^u(q, \epsilon))} \neq X^u(q, \epsilon).
\end{equation}
In particular, if $y \in X^u(q, \epsilon)$, then $\varphi^{-mN}(y) \in X^u(q, \lambda^{-mN}\epsilon) \subseteq X^u(q, \epsilon)$. Furthermore, by the definition of the bracket, the fact that all points in $X^s(p, \epsilon_p) \times X^u(p, \epsilon_p)$ are synchronizing, and Lemma \ref{lem:rect-nbhd-implies-lcu-lcs}, we have that
\begin{enumerate}
\item $[ y , p ] \sim_{lcs} y$, 
\item $[ \varphi^{-mN}(y), p ] \sim_{lcs} \varphi^{-mN}(y)$, 
\item $[ y, p] \in X^u(p, \epsilon_p)$, and
\item $[\varphi^{-mN}(y), p] \in X^u(p, \epsilon_p)$.
\end{enumerate}
Define the set
\[
\Delta = \{ (\varphi^{-n}([\varphi^{-mN}(y), p]), mN, \varphi^{-n}([y, p]) ) \mid y \in X^u(q, \epsilon)  \}, 
\]
which we must show satisfies the conditions in the definition of locally contracting. 

Firstly, it follows from the four item list above that $\Delta \subseteq G^{lcs}(X, \varphi, P)\rtimes \Z$. In addition, the range and source are homeomorphisms onto their images because the bracket and $\varphi$ have this property. Hence $\Delta$ is a $G^{lcs}(X, \varphi, P)\rtimes \Z$-set.

Finally, we have that $r(\bar{\Delta}) \subseteq s(\Delta)$ and $r(\bar{\Delta}) \neq s(\Delta)$ using  
\[
r(\bar{\Delta})= \overline{\{ \varphi^{-n}([\varphi^{-mN}(y), p]) \mid y \in  X^u(q, \epsilon)  \}}
\]
and 
\[
s(\Delta)=\{ \varphi^{-n}([y, p])  \mid y \in X^u(q, \epsilon) \}
\]
and Equation (\ref{contractingEquation}) above. 

Thus the groupoid is locally contracting as required.
\end{proof}

\begin{corollary}
Suppose $(Y, f)$ is a mixing synchronizing system such that the conditions in \cite[Lemma 4.19]{DeeSto} hold and $G^{lcs}(X, \varphi, P)$ and $G^{lcu}(X, \varphi, P)$ are amenable. Then $C^*_r(G^{lcs}(X, \varphi, P) \rtimes \Z)$ and $C^*_r(G^{lcu}(X, \varphi, P) \rtimes \Z)$ are non-unital Kirchberg algebras (i.e., nuclear, simple, separable, and purely infinite).
\end{corollary}
\begin{proof}
This follows from the main result of \cite{AnaDel} (in light of our previous work in this section showing that the relevant groupoids are minimal, essential principle, and locally contracting).
\end{proof}

\section{Main results} \label{sec:Main}

\subsection{Groupoid results}

The main result of the paper is Theorem \ref{thm:main}. After it is proved, we will discuss the $C^\ast$-algebraic implications of it. We begin with a few lemmas.

\begin{lemma} \label{lem:one-to-one-Xu(P)}
Suppose $(X, \varphi)$ is a Smale space, $(Y, f)$ is a finitely presented system, $\pi: X \rightarrow Y$ is an almost one-to-one u-resolving map, $P \subseteq Y$ is a finite set of synchronizing periodic points such that $\pi^{-1}(p)$ is a single point for each $p\in P$, and $Q=\pi^{-1}(P)$. Then for each $y\in Y^u(P)$, the set $\pi^{-1}(y)=\{ x\}$ for some $x\in X^u(Q)$. 
\end{lemma}
\begin{remark}
It is worth noting that this result is stronger than the statement that $\pi|_{X^u(Q)}$ is one-to-one; the element $y$ has unique preimage with respect to the map $\pi : X \rightarrow Y$.
\end{remark}
\begin{proof}
Let $y\in Y^u(P)$. Then, by Lemma \ref{lem:periodic-h-implies-equals},  $y \in Y^u(p)$ for a unique $p\in P$. By assumption, there is unique $q\in X$ such that $\pi^{-1}(p)=\{ q \}$. Moreover, $q$ is a periodic point since $p$ is. Let $M>1$ be such that 
\[
\varphi^M(q)=q \hbox{ and }f^M(p)=p.
\]

Since $\pi$ is onto, $\pi^{-1}(y)$ is non-empty. Take $x\in \pi^{-1}(y)$ and form the sequence in $X$, $(\varphi^{M\cdot n}(x))_{n\in \N}$. Let $L$ be a limit point of this sequence. Then $\pi(L)=\lim_{k\rightarrow \infty} \pi( \varphi^{M\cdot n_k}(x))$ for some subsequence $(\varphi^{M\cdot n_k}(x))_{k\in \N}$. Using the definition of a factor map, $y\in Y^u(p)$, and $f^M(p)=p$, we get
\[
\pi(L)=\lim_{k\rightarrow \infty} \pi( \varphi^{M\cdot n_k}(x))=\lim_{k\rightarrow \infty} f^{M\cdot n_k}(\pi(x))=\lim_{k\rightarrow \infty}f^{M\cdot n_k}(y)=p.
\]
Hence $L=q$ because $\pi^{-1}(p)=\{q\}$. Using the fact that $X$ is compact and the fact that all limit points of the sequence are $q$, we have that $\lim_{n \rightarrow \infty} \varphi^{M\cdot n}(x)=q$. 

Next, we show $x \in X^u(q)$. Let $\epsilon>0$. Since $X$ is compact and $\varphi$ is continuous, $\varphi$ is uniformly continuous. Hence there exists $\delta>0$ such that 
\[
d(\varphi^r(a), \varphi^r(b))<\epsilon
\]
whenever $d(a, b)<\delta$ and $r=0, 1, \ldots, M-1$.

Since $\lim_{n \rightarrow \infty} \varphi^{M\cdot n}(x)=q$, there exists $N\in \N$ such that 
\[
d(\varphi^{M\cdot n}(x), q) < \delta \hbox{ for }n\ge N.
\]
Take $k \geq M \cdot N$. We have that
\[
d(\varphi^k(x), \varphi^k(q))=d(\varphi^{M\cdot l + r}(x), \varphi^{M\cdot l + r}(q))=d(\varphi^r(\varphi^{M\cdot l}(x)), \varphi^r(q))
\]
for some $l\ge N$ and $0\le r \le M-1$. Using our uniformly continuous condition above, we have that 
\[ d(\varphi^r(\varphi^{M\cdot l}(x)), \varphi^r(q))< \epsilon \]
because $d((\varphi^{M\cdot l}(x), q)< \delta$. Hence $x\in X^u(q)$.

Finally, since $\pi$ is u-resolving and $\emptyset \neq \pi^{-1}(y) \subseteq X^u(q)$, we have that $\pi^{-1}(y)$ is a single point in $X^u(q) \subseteq X^u(Q)$.
\end{proof}

\begin{lemma} \label{lem:closeImpliesStable}
Suppose $(X, \varphi)$ is a Smale space, $(Y, f)$ is a finitely presented system, $\pi: X \rightarrow Y$ is an almost one-to-one u-resolving map, $P \subseteq Y$ is a finite set of synchronizing periodic points such that $\pi^{-1}(p)$ is a single point for each $p\in P$, $Q=\pi^{-1}(P)$, $y_1$ and $y_2$ are in $Y^u(P)$, and $x_1$ and $x_2$ are the unique points in $X$ such that $\pi(x_1)=y_1$ and $\pi(x_2)=y_2$. If there exists $N\in \N$ such that
\begin{enumerate}
\item $f^N(y_2) \in Y^s(f^N(y_1), \epsilon_Y)$ and
\item $d_X(\varphi^N(x_1), \varphi^N(x_2)) < \epsilon_{\pi}$,
\end{enumerate}
then $x_1 \sim_s x_2$.
\end{lemma}
\begin{proof}
By assumption, we can form $[ \varphi^N(x_1), \varphi^N(x_2) ]$ and using items (1) and (2) in the statement have that
\begin{align*}
\pi( [ \varphi^N(x_1), \varphi^N(x_2) ])  & = [ \pi(\varphi^N(x_1)), \pi(\varphi^N(x_2))] \\
& = [ f^N(\pi(x_1)), f^N(\pi(x_2))] \\
& = [ f^N(y_1), f^N(y_2) ] \\
& = f^N(y_2) \\
& = \pi(\varphi^N(x_2)).
\end{align*}
Using the fact that $\varphi^N(x_2)$ is the unique preimage of $\pi(\varphi^N(x_2))$, we have that
\[ [ \varphi^N(x_1), \varphi^N(x_2) ] = \varphi^N(x_2). \]
So $\varphi^N(x_2) \in X^s(\varphi^N(x_1), \epsilon_{\pi})$ and in particular $x_1 \sim_s x_2$.
\end{proof}

\begin{lemma} \label{lem:locConIt}
Suppose $y_1 \sim_{lcs} y_2$ and $\gamma : Y^u(y_1, \epsilon_1) \rightarrow V$ is a stable local conjugacy. Then, for any $n\in \N$, $f^n \circ \gamma \circ f^{-n} : Y^u( f^n(y_1), \epsilon_1) \rightarrow f^n(V)$ is a stable local conjugacy between $f^n(y_1)$ and $f^n(y_2)$.
\end{lemma}
\begin{proof}
Firstly, since $d(f^{-1}(w), f^{-1}(z)) \leq \lambda d(w,z)$ whenever $z\in Y^u(w, \epsilon_1)$ and $0<\lambda<1$, we have that
\[
f^{-1}(Y^u(f(y_1), \epsilon_1) \subseteq Y^u(y_1, \lambda\epsilon_1) \subseteq Y^u(y_1, \epsilon_1).
\]
An induction argument implies that 
\begin{equation} \label{Eq:Ind}
 f^{-n}(Y^u(f^n(y_1), \epsilon_1)) \subseteq Y^u(y_1, \epsilon_1) 
\end{equation}
and hence $f^n \circ \gamma \circ f^{-n}$ is well-defined. Moreover, it is a homeomorphism onto its image because $\gamma$ is and both $f$ and $f^{-1}$ are global homeomorphisms.

Since $\gamma : Y^u(y_1, \epsilon_1) \rightarrow V$ is a stable local conjugacy, we have that
\begin{equation} \label{Eq:KnoLim} \lim_{k \rightarrow \infty} {\rm sup}_{z \in Y^u(y_1, \epsilon_1)} d( f^k(z), f^k(\gamma(z)))=0.
\end{equation}
Now, if $w \in Y^u(f^n(y_1), \epsilon_1)$, then (see above in Equation (\ref{Eq:Ind})) $f^{-n}(w)\in Y^u(y_1, \epsilon_1)$. Moreover, for $w \in Y^u(f^n(y_1), \epsilon_1)$, we have that
\[
 d( f^k(w), (f^k\circ f^n\circ \gamma\circ f^{-n})(w))  = d( (f^k\circ f^n)(f^{-n}(w)), (f^k\circ f^n)( \gamma(f^{-n}(w)))). \]
Using this equality and the known limit in Equation (\ref{Eq:KnoLim}), we get  
\[  \lim_{k \rightarrow \infty} {\rm sup}_{w \in Y^u(f^n(y_1), \epsilon_1)}  d( f^k(w), (f^k\circ f^n\circ \gamma\circ f^{-n})(w)) =0. \]
Finally, 
\[ (f^n \circ \gamma \circ f^{-n})(f^n(y_1))=f^n (\gamma(y_1))=f^n(y_2) \]
which implies that $f^n \circ \gamma \circ f^{-n}$ is a stable local conjugacy from $f^n(y_1)$ to $f^n(y_2)$ as required.
\end{proof}

\begin{lemma} \label{lem:LSCforN}
Suppose that $y_1 \sim_{lcs} y_2$ with local stable conjugacy given by $\gamma$. Given $0<\epsilon\leq \epsilon_Y$, there exists $\delta>0$ and $N\in \N$ such that for any $n\ge N$ the local stable conjugacy 
\[  f^n \circ \gamma \circ f^{-n} : Y^u( f^n(y_1), \delta) \rightarrow Y^u(f^n(y_2), \epsilon) \]
from $f^n(y_1)$ to $f^n(y_2)$ is well-defined (i.e., the image is contained in $Y^u(f^n(y_2), \epsilon)$) and is given by the map
\[ z \in Y^u( f^n(y_1), \delta) \mapsto [z, f^n(y_2)]. \]
\end{lemma}
\begin{proof}
Since $f$ is uniformly continuous, there exists $\tilde{\epsilon}>0$ such that
\[ d(f(x), f(y)) < \epsilon_{Y} \hbox{ whenever }d(x,y)< \tilde{\epsilon}. \]
We can assume that $0< \epsilon < \tilde{\epsilon}$. By Lemma \ref{lem:DclosedBraCts} the bracket map is uniformly continuous, so there exists $\delta>0$ such that 
\[ d_Y([a, b], [c, d]) < \epsilon \hbox{ whenever }d_{Y\times Y}((a, b), (c, d))< \delta \]
where $D \subseteq Y \times Y$ is the domain of the bracket.

Since $y_1 \sim_{lcs} y_2$, there exists $N\in \N$ such that
\[ f^n(\gamma(z)) \in X^s(f^n(z), \epsilon) \hbox{ for all }n \ge N. \]
In particular, $f^{n}(y_2) \in X^s(f^n(y_1), \epsilon)$ for all $n \ge N$.

Using the previous lemma (i.e., Lemma \ref{lem:locConIt}) and the fact that $\gamma$, $f^N$, and $f^{-N}$ are each uniformly continuous there exists $\delta>0$ (depending on $N$) such that
\[
\tilde{\gamma}:= f^N \circ \gamma \circ f^{-N} : X^u(f^N(y_1), \delta) \rightarrow X^u(f^N(y_2), \epsilon)
\]
defines a local conjugacy from $f^N(y_1)$ to $f^N(y_2)$. We must show that $\tilde{\gamma}(z)=[z, f^N(y_2)]$. Firstly, $\tilde{\gamma}(z) \in X^u(f^N(y_2), \epsilon)$ by construction. But if $z\in X^u(f^N(y_1), \delta)$, then $f^{-N}(z) \in X^u(y_1, \delta)$ and hence 
$\tilde{\gamma}(z)=f^N\gamma(f^{-N}(z)) \in X^s(z, \epsilon)$. So 
\[ \tilde{\gamma}(z) \in X^s(z, \epsilon) \cap X^u(f^N(y_2), \epsilon) = [ z, f^N(y_2)]. \]
This completes the proof for the case $n=N$. 

We now consider the case $n=N+1$ and will show that the same $\delta>0$ as in the case $n=N$ works in this case. 

By the previous lemma, $f \circ \tilde{\gamma}\circ f^{-1} : Y^u(f^{N+1}(y_1), \delta) \rightarrow Y^u(f^{N+1}(y_2))$ is a well-defined local stable conjugacy. We need to show that the map is equal to $z \mapsto [z, f^{N+1}(y_2)]$ and that the codomain can be taken to be $Y^u(f^{N+1}(y_2), \epsilon)$. 

Firstly, for $k\ge 0$ and $w\in Y^u(f^{N+1}(y_1), \delta)$, we have that 
\begin{align*} 
d(f^k (f \circ \tilde{\gamma}\circ f^{-1})(w), f^k(w)) & = d(f^{k+1}( [f^{-1}(w), f^N(x_2)]), f^{k+1} (f^{-1}(w)))
& < \epsilon 
\end{align*}
since $[f^{-1}(w), f^N(x_2)]\in X^s(f^{-1}(w), \epsilon)$. Hence $(f \circ \tilde{\gamma}\circ f^{-1})(w) \in X^s(w, \epsilon)$.

Next, again for $k\ge 0$ and $w\in Y^u(f^{N+1}(y_1), \delta)$, we consider 
\[ d(f^{-k}(f([f^{-1}(w), f^N(y_2)])), f^{-k}(f^{N+1}(y_2))). \] 
For $k=0$, we have that 
\[ d( f([f^{-1}(w), f^N(y_2)]), f(f^N(y_2)))< \epsilon_Y \hbox{ since }d( [f^{-1}(w), f^N(y_2)], f^N(y_2))< \epsilon \]
where we have used the fact that $f$ is uniformly continuous (see the first line of the proof).

Now, for $k\geq 1$, 
\begin{align*} d(f^{-k}(f([f^{-1}(w), f^N(y_2)])), f^{-k}(f^{N+1}(y_2))) & =d(f^{-k+1}([f^{-1}(w), f^N(y_2)]), f^{-k+1}(f^{N}(y_2))) \\
& < \epsilon 
\end{align*}
since $[f^{-1}(w), f^N(y_2)] \in Y^u(f^{N}(y_2), \epsilon)$.

In summary, $\tilde{\gamma}(w) \in X^u(f^{N+1}(y_2), \epsilon_Y) \cap X^s(f^{N+1}(w), \epsilon)$. It follows that 
\[ \tilde{\gamma}(w)= [w, f^{N+1}(y_2)]. \]

Finally, 
\[
d([w, f^{N+1}(y_2)], f^{N+1}(y_2))=d([w, f^{N+1}(y_2)], [f^{N+1}(y_1), f^{N+1}(y_2)])< \epsilon
\]
since $d(w, f^{N+1}(y_1))< \delta$. Using our work above, we therefore have that
\[
d(f^{-k}(f([f^{-1}(w), f^N(y_2)])), f^{-k}(f^{N+1}(y_2)))< \epsilon \hbox{ for all }k\ge 0.
\] 
Hence, $\tilde{\gamma}(w)=[w, f^{N+1}(y_2)] \in X^u(f^{N+1}(y_2), \epsilon).$ 

An induction argument can be used to complete the proof for all $n\ge N$. We omit the details.
\end{proof}

\begin{theorem} \label{thm:main}
Suppose that $(Y, f)$ is a mixing finitely presented system, $(X, \varphi)$ (a mixing Smale space) and $\pi: X \rightarrow Y$ are the minimal u-resolving extension of $(Y, f)$, $P \subseteq Y$ is a finite set of synchronizing periodic points such that $\pi^{-1}(p)$ is a single point for each $p\in P$, and $Q=\pi^{-1}(P)$. Then $\pi \times \pi : G^s(X, \varphi, Q) \rightarrow G^\lcs(Y, f, P)$ is an isomorphism of groupoids. Moreover, if $P$ is invariant under $f$, then $\pi\times \pi$ is an equivariant isomorphism (with respect to the $\Z$-action on $G^\lcs(Y, f, P)$ giving by $f\times f$). Moreover, the result holds with $s$ and $u$ interchanged. In particular, in this case, $\pi \times \pi : G^u(X, \varphi, Q) \rightarrow G^\lcu(Y, f, P)$ is an equivariant isomorphism where $\pi$ is the minimal $s$-resolving extension.
\end{theorem}

\begin{remark}
The proof is rather involved. More direct approaches are hindered by the fact that although $(\pi|_{X^u(Q)})^{-1}: Y^u(P) \rightarrow X^u(P)$ is a homeomorphism, there is no reason that it respects the metrics on $X$ and $Y$. On the other hand, $\pi|_{X^u(Q)}$ does respect these metrics (see for example the proof of \cite[Lemma 6.5]{DeeSto}). More specifically, if $(\pi|_{X^u(Q)})^{-1}: Y^u(P) \rightarrow X^u(P)$ was uniformly continuous, then one could use \cite[Proposition 4.2]{thomsen2010c} in the current proof. Unfortunately, there is no reason to believe that $(\pi|_{X^u(Q)})^{-1}$ is uniformly continuous.

It is worth noting explicitly where the condition that $\pi: X \rightarrow Y$ is the {\bf minimal} $u$-resolving extension is used in the proof. This occurs in the proof when we state that since $a_1 \neq a_2$, there exists $y^{\prime} \in Y^u(y)$ such that 
\begin{enumerate}
\item there exists $a^{\prime} \in X^u(a_1)$ such that $\pi(a^{\prime})=y^{\prime}$ and
\item $y^{\prime} \not\in X^u(a_2)$.
\end{enumerate}
The interested reader can see Lemma 4.15 in \cite{fisher13} for why the minimal condition implies this property. In fact, one can show that this property (along with $u$-resolving, almost one-to-one, etc) implies that $\pi$ is the minimal $u$-resolving extension. This fact is implicit in the proof of Theorem 1.2 in \cite{fisher13} (but we do not need it for our proof, so we will not expand further on this line of thought).

When $(Y, f)$ is a sofic shift then we proved the presented result in \cite[Theorem 4.6]{DeeStoShift}. The reader will notice that the proof in the general finitely presented case is much more involved than the sofic shift case.
\end{remark}

\begin{proof}
Before starting the proof, we note that $\pi$ is almost one-to-one, finite-to-one, and uniformly continuous. We can and will assume that the metrics on $X$ and $Y$ satisfy the properties in Section \ref{sec:MetExpSys}. Also, given $x\in X$ and $r>0$, let
\[ X(x, r)=\{ w\in X \mid d_X(w,x)<r \}.
\]
We use similar notation for open balls in $Y$.

By \cite[Lemma 6,8]{DeeSto}, the map $\pi \times \pi:  G^s(X, \varphi, Q) \rightarrow G^\lcs(Y, f, P)$ is a continuous open inclusion, so we need only show it is onto. 

Let $(y_1, y_2) \in G^\lcs(Y, f, P)$ and $\gamma: U \rightarrow V$ be a stable local conjugacy from $y_1$ to $y_2$ where $U$ and $V$ are open sets in $X^u(Q)$ that contain $y_1$ and $y_2$ respectively.

By Lemma \ref{lem:one-to-one-Xu(P)}, $\pi^{-1}(y_i)=\{ x_i \}$ for unique $x_i \in X^u(Q)$ with $i=1, 2$. We must show that $x_1 \sim_s x_2$. 

Since $X$ is compact, by passing to subsequences, we can assume that 
\[ \lim_{k\rightarrow \infty}\varphi^{n_k}(x_i)=a_i \hbox{ for some }a_i \in X \hbox{ with }i=1, 2.\] 
If $a_1=a_2$, then there exists $k \in \N$ such that
\begin{enumerate}
\item $f^{n_k}(y_2) \in Y^s(f^{n_k}(y_1), \epsilon_Y)$ and
\item $d_X(\varphi^{n_k}(x_1), \varphi^{n_k}(x_2)) < \epsilon_{\pi}$,
\end{enumerate}
where for the first item we have used the fact that $y_1 \sim_{lcs} y_2$. These two properties are exactly the ones in the statement of Lemma \ref{lem:closeImpliesStable}. Applying this lemma gives $x_1 \sim_s x_2$. 

So we are done, unless $a_1 \neq a_2$. Using the fact that $\pi$ is a factor map, we have that 
\[
\lim_{k\rightarrow \infty}f^{n_k}(y_1)=\pi(a_1) \hbox{ and }\lim_{k\rightarrow \infty}f^{n_k}(y_2)=\pi(a_2)
\]
However, $y_1 \sim_{lcs} y_2$, so $\lim_{n\rightarrow \infty} d(f^n(y_1), f^n(y_2)) =0$. Hence, $\pi(a_1)=\pi(a_2)$; we denote this common value by $y\in Y$. 

Again, using the fact that $Y$ is compact, there is a subsequence $f^{-m_l}(y)$ converging to some $\bar{y}\in Y$. Since $\pi$ is finite-to-one, we have that $\pi^{-1}(\bar{y})=\{ \bar{x}_1, \cdots, \bar{x}_{\bar{N}} \}$ for some $\bar{N}\in \N$ and $\bar{x}_1, \cdots, \bar{x}_{\bar{N}}$ in $X$. Take $\epsilon>0$ such that 
\begin{enumerate}
\item $\epsilon<\frac{\epsilon_{\pi}}{2}<\epsilon_X$ where $\epsilon_{\pi}$ is as in Lemma \ref{lem:EpsPi},
\item the collection $\{ X(\bar{x}_i, \epsilon)  \}_{i=1}^{\bar{N}}$ is pairwise disjoint, and
\item $\epsilon< \epsilon_Y$
\end{enumerate}

By \cite[Lemma 2.5.9]{MR3243636}, there exists $\eta>0$ such that
\[
\pi^{-1}(Y(\bar{y}, \eta)) \subseteq \cup_{i=1}^{\bar{N}} X(\bar{x}_i, \epsilon)
\]
where we note that, by our assumption on $\epsilon$, this is a disjoint union. By Lemma \ref{lem:LSCforN}, there exists $\delta_{\gamma}>0$ and $N\in \N$ such that for any $n\ge N$ the local conjugacy 
\[  f^n \circ \gamma \circ f^{-n} : Y^u( f^n(y_1), \delta_{\gamma}) \rightarrow Y^u(f^n(y_2), \epsilon) \]
from $f^n(y_1)$ to $f^n(y_2)$ is well-defined (i.e., the image is contained in $Y^u(f^n(y_2), \epsilon)$) and is given by the map
\[ z \in Y^u( f^n(y_1), \delta_{\gamma}) \mapsto [z, f^n(y_2)]. \]
By taking the minimum, we can assume that $\eta \le \delta_{\gamma}$.

Since $\pi$ is uniformly continuous, there exists $0<\delta< \epsilon$ such that
\[
d_Y(\pi(w), \pi(z))< \frac{\eta}{3} \hbox{ whenever }d_X(w,z)< \delta.
\]
Since $a_1 \neq a_2$ by \cite[Lemma 4.15]{fisher13}, there exists $y^{\prime} \in Y^u(y)$ such that 
\begin{enumerate}
\item there exists $a^{\prime} \in X^u(a_1)$ such that $\pi(a^{\prime})=y^{\prime}$ and
\item $y^{\prime} \not\in \pi(X^u(a_2))$.
\end{enumerate}
Using these two items and the fact that $f^{-m_l}(y)$ converges to $\bar{y}$, we can take $L$ large enough so that 
\begin{enumerate}
\item $f^{-m_{L}}(y) \in Y(\bar{y}, \frac{\eta}{3} )$, 
\item $f^{-m_{L}}(y^{\prime}) \in Y^u(f^{-m_{L}}(y), \frac{\eta}{3})$, and
\item $\varphi^{-m_L}(a^{\prime}) \in X^u\left(\varphi^{-m_L}(a), \frac{\delta}{2A_X}\right)$.
\end{enumerate}
where $A_X \ge 1$ is as in Lemma \ref{lem:AXconstant}.

To simplify notation, we replace $f^{-m_L}(y)$ with $y$, $f^{-m_L}(y^{\prime})$ with $y^{\prime}$, $\varphi^{-m_L}(a^{\prime})$ with $a^{\prime}$, and $\varphi^{-m_{L}}(a_i)$ with $a_i$ for $i=1,2$. Also, we can still assume that $\varphi^{n_k}(x_i)$ converges to $a_i$ (by taking $k$ large and shifting the index by $-m_L$). 

In summary, after this change of notation, we have that
\begin{enumerate}
\item[(a)] $y \in Y(\bar{y}, \frac{\eta}{3} )$, 
\item[(b)] $y^{\prime} \in Y^u(y, \frac{\eta}{3})$, and
\item[(c)] $a^{\prime} \in X^u(a, \frac{\delta}{2A_X})$.
\end{enumerate}
Take $k$ large enough, so that
\begin{enumerate}
\item[(I)] $n_k \ge N$,
\item[(II)] $d_X(\varphi^{n_k}(x_i), a_i)< \frac{\delta}{2A_X}$ for $i=1, 2$,
\item[(III)] $f^{n_k}(y_1) \in Y^s(f^{n_k}(y_2), \frac{\eta}{3})$, and
\item[(IV)] $f^{n_k}(\gamma(z)) \in Y^s(f^{n_k}(z),  \frac{\eta}{3})$ for all $z \in Y^u(y_1, \eta)$.
\end{enumerate}
We note that by Lemma \ref{lem:LSCforN} and the fact that $n_k\ge N$, the map 
\[ \chi:=f^{n_k} \circ \gamma \circ f^{-n_k}: Y^u(f^{n_k}(y_1), \eta) \rightarrow Y^u(f^{n_k}(y_2), \epsilon)
\] 
is a stable local conjugacy from $f^{n_k}(y_1)$ to $f^{n_k}(y_2)$.

Next
\begin{align*}
d(\varphi^{n_k}(x_1), a^{\prime}) & \le   d(\varphi^{n_k}(x_1), a_1) + d(a_1, a^{\prime}) \\
& <  \left( \frac{\delta}{2A_X} + \frac{\delta}{2A_X}\right) \\
& < \frac{\delta}{A_X} \\
& < \delta \hbox{ since }A_X\ge 1 \\
& < \epsilon.
\end{align*}
So we can form $[a^{\prime}, \varphi^{n_k}(x_1)]$. Moreover, using Lemma \ref{lem:AXconstant},
\[
d_X(\varphi^{n_k}(x_1), [a^{\prime}, \varphi^{n_k}(x_1)]) \le A_X d(\varphi^{n_k}(x_1), a^{\prime}) < \delta.
\]
Hence $[a^{\prime}, \varphi^{n_k}(x_1)] \in X^u(\varphi^{n_k}(x_1), \delta)$ and by our choice of $\delta$, 
\[
\pi( [a^{\prime}, \varphi^{n_k}(x_1)] ) \in Y^u\left(f^{n_k}(y_1), \frac{\eta}{3}\right).
\]
So we can apply $\chi$ to this element. Since $\pi$ is a factor map and using Lemma \ref{lem:EpsPi} (which can be applied since $d(\varphi^{n_k}(x_1), a^{\prime})< \epsilon < \epsilon_{\pi}$),
\[
\chi(\pi( [a^{\prime}, \varphi^{n_k}(x_1)] ))=f^{n_k} ( \gamma (f^{-n_k}([ y^{\prime}, f^{n_k}(y_1)])))
\]
Moreover, using item (IV) above with $z=f^{-n_k}([ y^{\prime}, f^{n_k}(y_1)])$, we have that 
\[ \chi(\pi( [a^{\prime}, \varphi^{n_k}(x_1)] )) \in Y^s\left([ y^{\prime}, f^{n_k}(y_1)] , \frac{\eta}{3}\right).
\] 
The fact that 
\[
d_X(a^{\prime}, [a^{\prime}, \varphi^{n_k}(x_1)]) < \delta
\]
implies that
\[
d_Y(y^{\prime}, [y^{\prime}, f^{n_k}(y_1)])< \frac{\eta}{3}.
\]
Items (a) and (b) imply that $d_Y(y^{\prime}, y)< \frac{\eta}{3}$ and $d_Y(y, \bar{y})< \frac{\eta}{3}$. Using these three inequalities and the triangle inequality, we have that
\[ d_Y(\chi(\pi( [a^{\prime}, \varphi^{n_k}(x_1)], \bar{y}) < \eta.
\]
It is also worth noting that $\chi(\pi( [a^{\prime}, \varphi^{n_k}(x_1)])) \in Y^u(f^{n_k}(y_2), \epsilon) \subseteq Y^u(P)$.

Since $\chi(\pi( [a^{\prime}, \varphi^{n_k}(x_1)])) \in Y^u(P)$, it has unique preimage under $\pi$ (see Lemma \ref{lem:one-to-one-Xu(P)}). Furthermore, since $\chi(\pi( [a^{\prime}, \varphi^{n_k}(x_1)]))\in Y(\bar{y}, \eta)$, we have that the unique preimage of $ \chi(\pi( [a^{\prime}, \varphi^{n_k}(x_1)]))$ is in $\cup_{i=1}^{\bar{N}} X(\bar{x}_i, \epsilon)$. The sets in this union are pairwise disjoint, so the preimage is in exactly one of them. 

In addition, since $\pi(a_2)=y$ and $y \in Y(\bar{y}, \eta)$ (see item (a) above), we have that 
\[ a_2 \in \pi^{-1}(y) \subseteq \cup_{i=1}^{\bar{N}} X(\bar{x}_i, \epsilon).
\] 
By reordering we can assume that $a_2 \in X(\bar{x}_2, \epsilon)$ (and $a_2 \not\in X(\bar{x}_j, \epsilon)$ for $j\neq 2$). Using item (II) above, 
\[
d_X(\varphi^{n_k}(x_2), \bar{x}_2) \le d_X(\varphi^{n_k}(x_2), a_2) + d_X(a_2, \bar{x}_2) < \frac{\delta}{2A_X} + \epsilon < \epsilon_{\pi} 
\] 
leading to $\varphi^{n_k}(x_2)\in X(\bar{x}_2, \epsilon_{\pi})$. 

Let $i$ be the unique index such that 
\[ 
c=\pi^{-1}( \chi(\pi( [a^{\prime}, \varphi^{n_k}(x_1)]))) \in X(\bar{x}_i, \epsilon).
\] 
In particular, $[\bar{x}_i, c ]$ is well-defined and Lemma \ref{lem:EpsPi} can be used. Applying $\pi$ to $[\bar{x}_i, c]$, we get
\begin{align*}
\pi ( [ \bar{x}_i, c ]) & = [ \pi(\bar{x}_i, \pi(c)] \hbox{ by Lemma \ref{lem:EpsPi}} \\
& = [ \bar{y}, \pi(c) ] \\
& = [ \bar{y}, f^{n_k}(y_2) ] \hbox{ since }\pi(c) \in X^u(f^{n_k}(y_2), \epsilon) \\
& = [ \pi(\bar{x}_2), \pi(\varphi^{n_k}(x_2))] \\
& = \pi ([ \bar{x}_2, \varphi^{n_k}(x_2)])  
\end{align*}
where to get the last line, we have applied Lemma \ref{lem:EpsPi} (which can be used because $\varphi^{n_k}(x_2)\in X(\bar{x}_2, \epsilon_{\pi})$).
Since $\pi$ is one-to-one on $X^u(P)$, it follows that $[ \bar{x}_i, c ]=[ \bar{x}_2, \varphi^{n_k}(x_2) ]$. We know that $[ \bar{x}_2, \varphi^{n_k}(x_2) ] \in X(\bar{x}_2, \epsilon)$ and $[ \bar{x}_i, c ] \in X(\bar{x}_i, \epsilon)$, so we must have that $i=2$. 

Hence $c\in X(\bar{x}_2, \epsilon)$ and we can form $[ c, a_2]$. Moreover, 
\[ d_X(c, a_2)\le d_X(c, \bar{x}_2) + d_X(\bar{x}_2, a_2)< 2 \epsilon < \epsilon_{\pi},
\] 
so Lemma \ref{lem:EpsPi} can be applied. We have that
\begin{align*}
\pi([c, a_2]) & = [ \pi(c), \pi(a_2) ] \hbox{ by Lemma \ref{lem:EpsPi}} \\
& = [ \chi(\pi( [a^{\prime}, \varphi^{n_k}(x_1)])), y ] \\
& = [ \pi(a^{\prime}), y ] \hbox{ since }\chi(\pi( [a^{\prime}, \varphi^{n_k}(x_1)]))\in X^s(\pi(a^{\prime}), \epsilon_Y) \\
& = [ y^{\prime}, y] \\
& = y^{\prime} \hbox{ by item (b) above.}
\end{align*}
However, $[c, a_2] \in X^u(a_2)$ and $y^{\prime} \not\in \pi(X^u(a_2))$. This is a contradiction.

\end{proof}

\subsection{$C^\ast$-algebraic results}
Using the main result of the previous section (i.e., Theorem \ref{thm:main}) we obtain a number of $C^*$-algebraic results. 
\begin{corollary}
Suppose $(Y, f)$ is a mixing finitely presented system with $\pi_s : (X_s, \varphi_s) \rightarrow (Y, f)$ its minimal $s$-resolving extension and $\pi_u : (X_u, \varphi_u) \rightarrow (Y, f)$ its minimal $u$-resolving extension (notice that both $(X_s, \varphi_s)$ and $(X_u, \varphi_u)$ are mixing Smale spaces). Moreover, let $P_s \subseteq Y$ be a finite set of synchronizing periodic points such that $\pi_s^{-1}(p)$ is a single point for each $p\in P_s$, and $Q_s=\pi^{-1}(P_s)$ and likewise with $u$ in place of $s$. Then 
\[
C^*(G^s(X_u, \varphi_u, Q_u)) \cong C^*(G^\lcs(Y, f, P_u)) \hbox{ and } C^*(G^u(X_s, \varphi_s, Q_s)) \cong C^*(G^\lcu(Y, f, P_u)).
\]
Moreover, if $P_s$ and $P_u$ are invariant under $f$, then
\begin{align*}
C^*(G^s(X_u, \varphi_u, Q_u)) \rtimes \Z & \cong C^*(G^\lcs(Y, f, P_u))\rtimes \Z  \hbox{ and } \\
C^*(G^u(X_s, \varphi_s, Q_s))\rtimes \Z & \cong C^*(G^\lcu(Y, f, P_s))\rtimes \Z.
\end{align*}
\end{corollary}
\begin{proof}
The first and last statements follow directly from the previous result where it is the fact that the isomorphism is equivariant that implies the ``moreover" part.
\end{proof}

Based on the previous theorem, properties such as finite nuclear dimension, real rank zero, stable rank one, etc that hold for the stable and unstable of a mixing Smale space also hold for the stable synchronizing heteroclinic algebra and unstable synchronizing heteroclinic algebra associated to a finitely presented system; the same is true for the crossed products by $\Z$. We can also obtain results about the synchronizing ideal and the homoclinic algebra. As a sample theorem, we have the following result:

\begin{corollary}
Suppose $(Y, f)$ is a mixing finitely presented system. Then 
\begin{enumerate}
\item the synchronizing ideal of $(Y, f)$ has finite nuclear dimension and has real rank zero.
\item If $(Y, f)$ has only finitely many non-synchronizing points, then the homoclinic algebra of $(Y, f)$ has finite nuclear dimension and real rank zero.
\end{enumerate}
\end{corollary}
\begin{proof}
Using the notation of the previous theorem, the synchronizing ideal of $(Y, f)$ is Morita equivalent to $C^*(G^s(X_u, \varphi_u, Q_u)) \otimes C^*(G^u(X_s, \varphi_s, Q_s))$ by \cite[Theorem 6.10 part (i)]{DeeSto} and then the results follows from known results about Smale space $C^\ast$-algebras, see in particular \cite{MR4069199, MR3766855}.

For the homoclinic algebra, using the fact that there are only finitely many non-synchronizing points (see \cite[Theorem 6.11]{DeeSto}) we have the following short exact sequence
\[
0\rightarrow \mathcal{I}_\text{sync}(Y, f) \rightarrow A(Y,f) \rightarrow \C^n \rightarrow 0
\]
where $\mathcal{I}_\text{sync}(Y, f)$ is the synchronizing ideal, $A(Y,f)$ is the homoclinic algebra, and $n\in \N$ is the number of non-synchronizing points. That $A(Y,f)$ has finite nuclear dimension follows since both $\mathcal{I}_\text{sync}(Y, f)$ and $\C^n$ have this property and \cite[Proposition 2.9]{WinterZac:dimnuc}. In a similar way, $A(Y,f)$ has real rank zero because $\mathcal{I}_\text{sync}(Y, f)$ and $\C^n$ do and $K_1(\C^n) \cong \{ 0 \}$, see \cite{LinRor}.
\end{proof}

\begin{corollary}
Suppose $(Y, f)$ is a mixing almost Smale space with $\pi : (X, \varphi) \rightarrow (Y, f)$ almost one-to-one, $s$-resolving, $u$-resolving where $(X, \varphi)$ is a Smale space. Moreover, let $P \subseteq Y$ is a finite set of synchronizing periodic points such that $\pi^{-1}(p)$ is a single point for each $p\in P$, and $Q=\pi^{-1}(P)$. Then 
\[
C^*(G^s(X, \varphi, Q)) \cong C^*(G^\lcs(Y, f, P)) \hbox{ and } C^*(G^u(X, \varphi, Q)) \cong C^*(G^\lcu(Y, f, P)).
\]
In addition, the synchronizing ideal of $(Y, f)$ is Morita equivalent to the homoclinic algebra of $(X, \varphi)$. Moreover, if $P$ is invariant under $f$, then
\begin{align*}
C^*(G^s(X, \varphi, Q)) \rtimes \Z & \cong C^*(G^\lcs(Y, f, P))\rtimes \Z \hbox{ and }  \\
C^*(G^u(X, \varphi, Q))\rtimes \Z & \cong C^*(G^\lcu(Y, f, P))\rtimes \Z.
\end{align*}
\end{corollary}
\begin{proof}
The map $\pi: (X, \varphi) \rightarrow (Y, f)$ is both the minimal $s$-resolving extension and the minimal $u$-resolving cover for $(Y, f)$. The first and last statements then follow directly from the previous result. For the ``in addition" statement, we have
\begin{align*}
\mathcal{I}_\text{sync}(Y, f) &  \sim_{ME} C^*(G^\lcu(Y, f, P)) \otimes C^*(G^\lcs(Y, f, P)) \\
& \cong C^*(G^u(X, \varphi, Q)) \otimes C^*(G^s(X, \varphi, Q)) \\
& \sim_{ME} H(X, \varphi)
\end{align*}
where $\sim_{ME}$ denotes Morita equivalence, and we have used \cite[Theorem 3.1]{putnam_1996} and \cite[Theorem 6.1 part (i)]{DeeSto}.
\end{proof}

\end{document}